\documentclass[
article,
10pt,oneside]{amsart}
\usepackage{amsmath,amsthm,amsfonts}
\usepackage{amssymb,latexsym}
\usepackage[cmtip,all]{xy}
\usepackage{graphicx}
\usepackage{caption}
\usepackage{hyperref}
\usepackage{cleveref}

\newtheorem{theorem}{THEOREM}[section]
\newtheorem{definition}[theorem]{Definition}
\newtheorem{remark}[theorem]{Remark}

\newtheorem{lemma}[theorem]{Lemma}

\title{Admissible Complexes for the Projective X-Ray Transform over a Finite Field}
\author{David V. Feldman \\ University of New Hampshire
\\ \\ \and \\ \\
Eric L. Grinberg\\ University of Massachusetts, Boston}

\address{Department of Mathematics \\ 
               University of New Hampshire \\  
               Kingsbury Hall \\
               Durham NH 03824\\ USA}
\email{\href{mailto:david.feldman@umb.edu}{david.feldman@unh.edu} }                

\address{Department of Mathematics \\
             University of Massachusetts \\
             100 Morrissey Boulevard \\
             Boston, MA 02125 \\
             USA}
\email{\href{mailto:eric.grinberg@umb.edu}{eric.grinberg@umb.edu} }         

\date{}
\def\Cal{\mathcal}
\def\fq{\mathbb F_q}
\def\fp3{\mathbb F_q P^3}
\def\are3{\mathbb R^3}
\def\Points{{\Cal P}}
\def\Lines{{\Cal L}}
\def\Quadrics {{\Cal Q}}

\begin{document}



\begin{abstract}
We consider the X-ray transform in a projective space over a finite field. It is well known (after E.~Bolker) that  this transform is injective. We formulate an analog of I.M.~Gelfand's {\it admissibility problem} for the Radon transform, which asks for a classification of all minimal sets of lines for which the restricted Radon transform is injective. The solution involves doubly ruled quadric surfaces.
\end{abstract}

\maketitle

\section{Introduction}

We consider a problem in integral geometry for projective spaces over finite fields. This can be manifested using S.S.~Chern's double fibration diagram  introduced by S.~Helgason in the context of homogeneous spaces and by  V.~Guillemin and S.~Sternberg in the context of microlocal analysis:

\[
\begin{xy}
\begin{large}
\xymatrix{
     & Z \ar[dl]_{\pi}\ar[dr]^{\rho} & \\
      X  & & Y}
\end{large}
\end{xy}
\qquad\qquad\qquad
\]

\vskip 0.6truein

\noindent
Here $X$ and $Y$ are finite sets and $Z$ is an {\em incidence relation} which is a subset of the Cartesian product $X \times Y$.  Our starting point is the point-line diagram in projective space. Let $\fq$ be a finite field with $q$ elements. Let $X = \fp3$ be the projective $3$-space over $\fq$, let $Y$ be the set of projective lines in $\fp3$ and let $Z$ be the collection of pairs $(x, \ell )$, where $x$ is a point in $\fp3$, $\ell$ is a projective line in $\fp3$, and $x$ lies in $\ell$. The maps $\pi$ and $\rho$ are inherited from the Cartesian projection maps. These induce mappings on function spaces:
$$
\pi^* : C(X) \longrightarrow C(Z) \quad ; \quad \rho_* : C(Z) \longrightarrow C(Y),
$$
where $\pi^*$ is the pullback of $\pi$, while $\rho_*$ is the pushforward, or summation over the fiber, map.

The Radon transform attached to this diagram (and manifestation of push-forward) is $R \equiv \rho_* \pi^*$. This an \lq\lq integral transform" taking point functions $f(x)$ to line  functions $Rf( \ell )$:
$$
Rf( \ell ) \equiv \sum_{ x \in \ell } f(x).
$$
The number of lines passing through a point $x$ is independent of $x$ (and greater than $1$) and the number of lines passing through two distinct points $x_1,x_2$ is $1$, hence independent of the pair $x_1 \ne x_2$. Thus this diagram satifies axioms introduced by E.~Bolker \cite{Bol} in the 1970s, which became known as the {\em Bolker condition}, and hence the transform is injective, with a simple inversion formula. 

\begin{theorem}[Bolker]
Assume that the double fibration diagram satisfies the following two conditions:
$$
\begin{aligned}
\bullet \qquad & \# \left( \pi^{-1}(x) \right) \qquad\qquad\quad  = \alpha ,   \forall x & \quad  
\qquad \textrm {(uniform count of lines through each point)}
 \\
\bullet \qquad & \#  \left(\pi^{-1}(x_1) \cap \pi^{-1}(x_2) \right) = \beta ,   
 \forall   x_1 \ne x_2  &
\textrm {(uniform count of lines through each point pair),}
\end{aligned}
$$
for constants $\alpha , \beta$, with  $0 \ne \alpha \ne  \beta$.
Then the Radon transform associated with the diagram is invertible, 
with an explicit inversion formula.
\end{theorem}
One proof of this result uses a natural basis on the space of functions in $X$ to present the associated Radon transform, composed with its dual, as a rank one perturbation of a diagonal matrix. The Bolker Condition holds for many classical geometries, e.g., 
whenever there is a doubly transitive group of symmetries acting. But there are also
many situations in which the condition does not hold. It seems plausible
that a somewhat more general condition may be workable, 
one where the cardinality of the set of lines through
a pair of points can vary within a multi-element collection of values, leading to
a banded matrix for $R^rR$. 
In the admissibility analysis below we will encounter a corresponding matrix which has more bands, but which is still sufficiently simple for complete analysis.

In the spirit of I.M.~Gelfand's work on the Radon transform in the continuous category, we note that the inversion problem is overdetermined (there are more lines than points), and ask the admissibility question: what are minimal collections of lines (subsets of $Y$) for which the restricted Radon transform is injective? As in the continuous work of Gelfand and collaborators we seek a description of these minimal data sets which is systematic and geometric in nature. Some results on admissibility in the finite category may be found, e.g., in \cite{BGK, Gr1, Gr2}. 

Of course, there are many variations on this double fibration diagram, where \lq\lq points" need not be points, and \lq\lq lines" need not be lines. Also, the manifestation of \lq\lq integration" over the fiber can be varied. In a number of classical cases the Bolker condition is satisfied and the associated Radon transform is easily inverted. On the other hand, in  many other cases these conditions are not satisfied and the invertibility properties are more involved. Our admissibility investigations for the line, or X-Ray  transform in $\fp3$ will lead us to another double fibration, where $X$ is the set of lines in $\fp3$ and $Y$ is a collection of doubly ruled quadric surfaces in $\fp3$, which may be viewed as special families of lines. Moreover, our push-forward operation will involve not counting measure, but rather a signed measure. The associated transform will enable us to characterize admissible sets of lines for the first transform.  The analysis, initially considered for the smallest field $\mathbb Z_2$, turns out to hold uniformly for all finite fields, regardless of structure.

\bigskip

\section{The X-Ray transform in Projective Space over the Field $\mathbb F_q$}

Recall for the moment the classical problem considered by J.~Radon \cite{Radon1917}, that
of determining a (reasonable) function $f(x)$ in $\are3$ from its integrals over lines. Now the manifold of lines in $\are3$ is of dimension four, one more than the dimension of the Euclidean space $\are3$. Hence the Radon inversion problem is overdetermined by one dimension, and one presumes that some three dimensional family of lines should suffice for Radon inversion. And indeed it turns out that the family of lines meeting a fixed curve $\Gamma$ (for a large class of possible curves $\Gamma$) or the family of lines tangent to a surface $\Sigma$ (for a rather restrited class of possible surfaces $\Sigma$) will work \cite{Gel1, Gel2, Kir}.
Similarly, the projective ($3$)-space over the finite field $\mathbb F_q$ has cardinality $O(q^3)$, while the collection of projective lins in this projective space has cardinality $O(q^4)$. The projective line transform in $\fp3$ satisfies the Bolker Condition, and is hence invertible. Linear algebraic considerations suggest that a collection of lines of
cardinality $O(q^3)$ should suffice for Radon inversion. This is indeed the case and will be explored more precisely below.

The $3$-dimensional projective space $\fp3$ over a finite field ${\Bbb F}_q$ with $q$ elements has
$\Points := \frac{q^4-1}{q-1} = q^3+q^2+q+1$ points and $\Lines := \frac{\Points(\Points-1)}{(q+1)q}=(q^2+1)(q^2+q+1)$ lines.
(Each line contains $q+1$ points.)
See \cite{Bol} 
\begin{definition}
 A {\bf line complex} is a collection of lines $L \subset \Lines$, which contains as many elements as there are points in $\fp3$: $ |L| = | \Points |$. 
\end{definition}

A line complex  $L$ determines a restricted integration  map, which is a linear transformation from functions of points in $\fp3$ to functions of lines in $L$: 
$$ 
T_L:C( \Points ) \rightarrow C(  L ) 
$$
where 
$$
T_L(f)(\ell)=\sum_{\rho\in \ell} f(\rho).
$$

This is an analog of the {\em Radon} or {\em X-Ray} transform.  Injectivity for $T_L$ affords the possibility of determining every function $f$
from its \lq\lq integrals" over lines in $L$.  

\begin{definition}
We say that a line complex $L$ is {\em admissible } if the restricted $X-ray$ transform  $T_L$ is injective. 
\end{definition}

Bolker's Theorem  implies that the full X-Ray transform is injective, and elementary rank considerations indicate that, as a consequence, admissible complexes must exist. I.M.~Gelfand pioneered the study of admissibility in integral geometry and, together with his collaborators, developed the subject deeply in a series of papers.

The aim at hand is to characterize the admissible line complexes among all complexes. Gelfand's original approach in the continuous category was to use the {\em range characterization} of the Radon transform, i.e., which functions of {\em lines } are integrals of functions of {\em points}. A necessary set of conditions, which Gelfand calls the {\em Cavalieri Conditions} is manifested in our context by certain doubly ruled surfaces $\Sigma$. Each such surface may be presented as a disjoint union of lines, Bolker's  {\em local spread} of lines, in more than one way. If a line function $f( \ell )$ in the range of the Radon transform is averaged over a spread of lines, the  average value is simply the normalized total mass of the Radon pre-image of $f ( \ell )$ over the surface $\Sigma$, and is hence independent of the spread chosen. Thus we obtain a linear condition for each surface $\Sigma$ and each pair of spreads on $\Sigma$. We note that Bolker \cite{Bol} constructed a family of global spreads on $\fp3$ using complex structures on the plane $\fq^2$, and showed, using a clever representation theory argument, that the corresponding Cavalieri conditions characterize the range of the Radon transform. Moreover, Bolker anticipated the role of ruled quadric surfaces.

\centerline{
\includegraphics[width=60mm]
{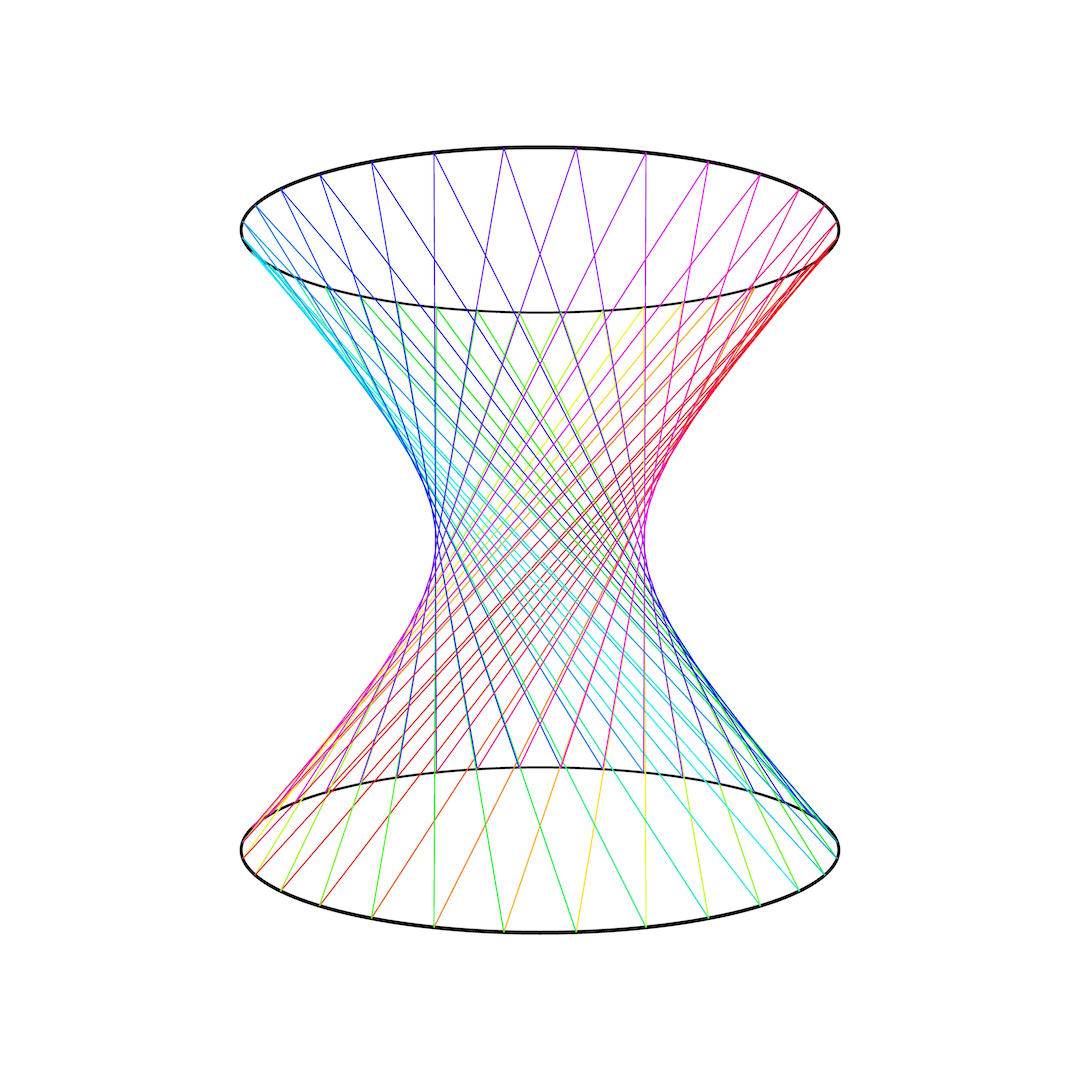}
}

\begin{center}
\includegraphics[width=0.60mm]{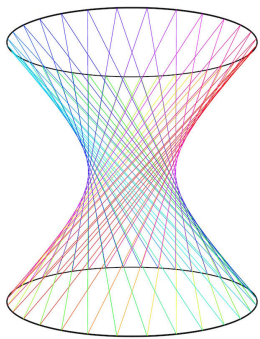}
\captionof{figure}{A ruled surface}
\end{center}

 In the continuous category the resulting conditions do not suffice, as not enough analytic information is obtained about the possible Radon pre-image. But in the present context we will show that these conditions suffice. In Bolker's jargon, {\em there are enough spreads}.

With the aim of  characterizing admissible line complexes for the X-Ray transform in $\fp3$ we introduce certain formal integral linear combinations of lines that we shall call {\bf DRQ}s, for {\bf doubly-ruled quadrics}. 
A DQR has the form 
$$
\sum_{i=0}^q m_i - \sum_{i=0}^q n_j ,
$$
where each line $m_i$ meets each line $n_j$ but no other intersections occur.  In particular, $\bigcup m_i = \bigcup n_j.$  The
lines lying on the surface $wx-yz=0$ support a DRQ, and up to collineation all DRQs have this form; this explains the
terminology.

An admissible line complex $L$ never supports any non-trivial linear combination of DRQs. Indeed, injectivity of $T_L$
entails surjectivity of $T_L$.  If $L$ supported $\sum_k \alpha_k Q_k$ with $Q_k=\sum_{i=0}^q m_{i,k} -\sum_{j=0}^q
n_{j,k}$, then for all $f$ we would have $\sum_i T_L(f)(m_{i,k}) = \sum_j T_L(f)(n_{j,k})$, a non-trivial linear
condition cutting down the dimension of the image of $T_L$.

Actually this necessary DRQ-avoidance condition also suffices:

\bigskip
\begin{theorem} Let $L$ be a line complex in $\fp3$. Then $L$ is admissible if and only if $L$ supports no Doubly Ruled Surface (DRQ) $\Sigma$.
\end{theorem}

The proof will follow from a series of lemmas.

\bigskip\noindent
Regarding $\Lines$ as the totality of lines in $\fp3$, we have the dual Radon transform, which is a linear transformation:
$$ 
T_\Lines^t:C(  \Lines ) \rightarrow  C( \Points ) 
$$
defined by setting, for a line function $g( \ell )$,
$$
T_\Lines^t g( \ell ) (p) =\sum_{ \{ \ell \, \vert \, p \in \ell \} } g( \ell ).
$$ 
We aim to show that DRQs span the kernel of $T_\Lines^t$.  
Then for any line complex that supports no linear combination
of DRQs, $T_L^t:C(  L ) \rightarrow  C( \Points )$ will have trivial kernel, thus inject and indeed biject,
thereby making $T_L$  itself an injection, as desired.

Since one already knows the injectivity of $ T_\Lines:C( \Points ) \rightarrow C(  \Lines ) $ from Bolker's theorem in the theory of the finite or combinatorial Radon transform, one knows that $T_\Lines^t$ surjects, and thus has a kernel of dimension $|\Lines|-|\Points|=q^4+q^2$.  Thus we aim to show that the span of the DRQs has the same dimension.

The lines of a DRQ determine it only up to sign. So for each set of $2(q+1)$ lines that do support a DRQ, choose, once
and for all, just one DRQ, $\sum_{i=0}^q m_i - \sum_{j=0}^q n_j$, and write $\Quadrics$ for the set of these.

Linear combinations of linear combinations yield linear combinations, so we have a linear transformation
from functions of DRQs to functions of lines:
$$ 
T_\Quadrics^t:C(  \Quadrics ) \rightarrow C(  \Lines )\ ,
$$
which is the adjoint of the map defined by 
$$
T_{\Quadrics}(g)(Q)= \sum_ig(m_i) 
- \sum_j g(n_j)
$$ 
for 
$$
Q=\sum_{i=0}^q m_i - \sum_{j=0}^q n_j\ .
$$

We may regard the dimension of the image of $T_\Quadrics^t$
 as the rank of a certain matrix $A$ with rows labelled by
lines and  columns labelled by DRQs, 
The various entries of this matrix carry the coefficients of the chosen DRQs.  
The square
matrix $B \equiv AA^t$, which we will call the {\bf Cavalieri matrix}, has rows and columns labelled by lines, 
has the same rank as $A$ 
and also has the advantage that it doesn't depend upon the enumeration of DRQs 
or the sign choices above.  
To describe the entries of $B$ first we need to count all the DRQs.

\begin{lemma} The projective space $\fp3$ contains {\bf triads}, that is,
triples of pairwise disjoint lines. Any disjoint pair of lines may be extended to a
triad in $ q(q+1)(q-1)^2$ ways.
\end{lemma}

\begin{proof}
We observed above that the number of lines in our projective space is $|\Lines|=(q^2+1)(q^2+q+1)$. 
Fix a point $p$ in $\fp3$, i.e., a line in $\mathbb F_q^4$. The projective lines in $\fp3$ which meet $p$ correspond to the $1$-dimensional vector subspaces of the quotient space ${\mathbb F}^4 / p$, and these form a projective plane. Hence there are $(q^3-1)/(q-1) \, = \, q^2 +q +1$ projective lines that meet $\ell$ at a point $p \in \fp3$. 

Now fix a line $\ell \subset  \fp3$. Then there are $q^2+q+1-1$ lines meeting $\ell$ precisely at the point $p \in \ell$, and there are $(q+1)(q^2+q)$ lines meeting $\ell$ at a point. Therefore, there are 
$$
| \Lines | - (q+1)(q^2+q) =(q^2+1)(q^2+q+1)-(q+1)(q^2+q)-1=q^4
$$ 
lines disjoint from $\ell$.  Given two disjoint lines, $(q+1)(q^2+q)$ lines meet
at least the first one at a point, the same number meet at least the second, but $(q+1)^2$ meet both.  Thus
$(q^2+1)(q^2+q+1)-2(q+1)(q^2+q)+(q+1)^2-2=q(q+1)(q-1)^2$ {\em miss} both
\end{proof}

\bigskip
\begin{lemma}  Any triad of mutually disjoint lines may be extended uniquely to a DRQ, and the number of DRQs equals 
$$
\frac{1}{2}(q^2+1)(q^2+q+1)q^4(q-1)\ .
$$
\end{lemma}

\begin{proof}
Consider any three disjoint lines $L_i, i=1,2,3$.  Fixing a point $p\in L_1$, the plane spanned by $p$ and
$L_2$ meets $L_3$ once, at $q$ say.  The line $pq$ meets each $L_i, i=1,2,3$ once, and every line that does so arises
this way.  The set of these lines bijects with the points on $L_1$, say, so we have $q+1$ lines, say $m_0,\ldots,m_q$.  No
two of the $m_i$'s meet or they would sit in a common plane which would force at least two of the $L_i$ to do the same,
thereby forcing them to meet.  Choose any three of the $m_i$ and repeat the construction to produce $q+1$ lines
$n_0,\ldots,n_q$ (extending our original choice of $L_i, i=1,2,3$). Three disjoint lines thereby extend to a unique
DRQ.

Given a DRQ, we have a choice of two disjoint families, and we can choose a sequence of three distinct lines from one
of those families in $(q+1)q(q-1)$ ways, for a total of $2(q+1)q(q-1)$ options.

As noted in the lemma above and its proof, the number of lines in $\fp3$ is$|\Lines|=(q^2+1)(q^2+q+1)$. Given a line $\ell$, $q^4$ lines will sit disjoint from it.  
Two disjoint lines may be extended to a triad in $q(q+1)(q-1)^2$ ways.  
We arrive at the number of DRQs by dividing
ways to pick an arbitrary sequence of three disjoint lines by ways of selecting them from a given DRQ:
$$\frac{(q^2+1)(q^2+q+1)\cdot q^4\cdot q(q+1)(q-1)^2}{2(q+1)q(q-1)}=\frac{1}{2}(q^2+1)(q^2+q+1)q^4(q-1)\ .
$$
\end{proof}
Now we can describe the entries of $B$.

\bigskip\noindent
\begin{lemma} The Cavalieri matrix $B=AA^t$ has the following entries
\begin{equation*}
b_{\ell_1 , \ell_2} = 
\begin{cases}
q^4(q^2-1) & \textrm{\qquad if \qquad $\ell_1,\ell_2$ are equal} \\
-q^3(q-1)  & \textrm{\qquad if \qquad $\ell_1,\ell_2$ meet at a pt} \\
q(q^2-1)   & \textrm{\qquad if \qquad $\ell_1,\ell_2$ are disjoint} 
\end{cases}
,
\end{equation*}
\end{lemma}

\begin{proof} The entry $b_{\ell_1,\ell_1}$ counts the number of DRQs that contain $\ell_1$; $-b_{\ell_1,\ell_2}$ counts the
number of DRQs that contain both $\ell_1$ and $\ell_2$ if they meet at a point; and $b_{\ell_1,\ell_2}$ counts the
number of DRQs that contain both $\ell_1$ and $\ell_2$ if they don't meet.  We must make these counts.

Multiplying the total number of DRQs by lines in a DRQ and dividing by the total number of lines gives (1).
Multiplying the total number of DRQs by the number of ordered pairs of crossing lines in a DRQ ($2(q+1)^2)$ and
dividing by the total number of all ordered pairs of crossing lines ($(q^2+1)(q^2+q+1)(q+1)(q^2+q)$) gives (2). (The
negative sign accounts for lines from opposite disjoint families.) Multiplying the total number of DRQs by the number
of ordered pairs of skew lines in a DRQ ($2(q+1)q)$ and dividing by the total number of all ordered pairs of skew
lines ($(q^2+1)(q^2+q+1)q^4$) gives (2).
\end{proof}

\bigskip\noindent{\bf Notation}  Two lines may sit in one of three relations:
they may coincide ($=$), meet at a point ($\times$) or not meet ($||$).
Given two lines $L_1$ and $L_2$ that sit in relation $r_3$
we shall write $M_{r_3,(r_1,r_2)}$ for the number of lines $L_3$ that sit in relation $r_1$ with respect to $L_1$
and relation $r_2$ with respect to $L_2$.  Note that if either $r_1$ or $r_2$ stand for $=$, we get a count of 1
and thus suppress the notation in the sequel.

In the same spirit, let us write the matrix coefficients of $B$ as: 
$$
b_=:=q^4(q^2-1) \quad ; \qquad  b_\times=-q^3(q-1) \quad ; \qquad b_{||}:=q(q^2-1).
$$

\bigskip
Now we come to the fact that lies at the technical heart of our work.

\bigskip
\begin{lemma} 
$B$ is an orthogonal projection scaled by $v \equiv q^2(q-1)(q+1)(q^2+q+1)$. The rank of  $B$ is $q^4+q^2$.
\end{lemma}
\begin{proof} 
Orthogonality will follow once we establish idempotence for the symmetric and thus self-adjoint matrix $(1/v)B$.
Write $C:=B^2$.
We need to calculate matrix coefficients  $c_=$, $c_\times$ and $c_{||}$ for $C$.
Clearly
$$
c_= =  (b_=)^2+M_{=,(\times,\times)}(b_\times)^2+M_{=,(||,||)}(b_{||})^2
$$
where  $M_{=,(\times,\times)}=(q+1)(q^2+q)$ , and $M_{=,(||,||)}=q^4$ .
Calculating
$$
c_= = q^6(q^2+q+1)(q+1)^2(q-1)^2\ 
$$
and indeed  $c_=  = vb_=.$

Similarly
$$c_\times=  2(b_=)(b_\times)
           +2M_{\times,(\times,||)}(b_\times)(b_{||})
            +M_{\times,(\times,\times)}(b_{\times})^2
            +M_{\times,(||,||)}(b_{||})^2
            $$
Here
$M_{\times,(\times,||)}=q(q^2+q)-q^2=q^3$, $M_{\times,(\times,\times)}=(q^2+q-1)+q^2$ and
$$
M_{\times,(||,||)}=(q^2+q+1)(q^2+1)-(1+1)-2M_{\times,(\times,||)}-M_{\times,(\times,\times)}=q^4-q^3.
$$
Calculating, one indeed finds $c_\times = -q^{10}+q^8+q^7-q^5 = vB_\times$.

With similar notation,
$$
   c_{||}=  2(b_=)(b_{||})
           +2M_{||,(\times,||)}(b_\times)(b_{||})
            +M_{||,(\times,\times)}(b_{\times})^2
            +M_{||,(||,||)}(b_{||})^2,            
$$
where
$$
M_{||,(\times,||)}=(q+1)(q^2+q)-(q+1)^2=(q-1)(q+1)^2,
$$
$M_{||,(\times,\times)}=(q+1)^2$, and         
$$
M_{||,(||,||)}=
          (q^2+p+1)*(q^2+1)-(1+1)-2M_{||,(\times,||)}-M_{||,(\times,\times)}=q^4-q^3-q^2+q.
$$
Calculating again, 
$$
c_{||}= q^9+q^8-q^7-2q^6-q^5+q^4+q^3= vB_{||},
$$ 
as desired.

Finally, the rank of $B$ equals the trace of $(1/v)B$, which equals

$$
b_=\cdot|\Lines|/v=
\frac{q^4(q^2-1)  (q^2+1)(q^2+q+1)}{(q^2(q-1)(q+1)(q^2+q+1)}=q^4+q^2,
$$ 
as desired.
\end{proof}

We have thus shown that the dimension of Cavalieri conditions attached to DRQs is precisely the quantity by which the number of lines exceeds the number of points in $\fp3$. This completes the proof of our characterization of admissible line complexes in $\fp3$.

\begin{remark}
The quantity $v$, non-zero eigenvalue of $B$, equal $q^2/2$ times the number DRQs that
contain a given point.
\end{remark}

\begin{remark}
Playing the same game, now not with $A$, but symmetrizing the incidence matrix
relating lines and points, (which will produce entries of $q+1, 1, 0$ corresponding to
relations $=$, $\times$, $||$), one gets a matrix (with rows and columns indexed by lines)
close to a scaled projection except for one eigenvector with a distinct eigenvalue.
The constant function on lines will have eigenvalue $(q+1)(q^2+q+1)$.  Orthogonal to
this one we'll have a scaled projection, eigenvalue $(q+1)q$ with rank $|\Lines|-1$.
\end{remark}

\end{document}